\documentclass[a4paper]{scrartcl}
\usepackage[utf8]{inputenc}
\usepackage{natbib}
\usepackage{amsthm, amssymb, amsmath}
\usepackage{bm}
\usepackage{dsfont}
\usepackage{enumerate}
\usepackage{color}
\usepackage{soul}
\usepackage{caption}
\usepackage{graphicx}
\usepackage[FIGTOPCAP]{subfigure}
\usepackage[unicode,breaklinks]{hyperref}
\definecolor{linkcolor}{RGB}{15,15,175}
\definecolor{alarmcolor}{RGB}{175,15,15}

\hypersetup{
    pdfauthor={Christian Sch\"afer},
    pdfcreator={Christian Sch\"afer},
    unicode=true,           
    pdftoolbar=true,        
    pdfmenubar=true,        
    pdffitwindow=false,     
    pdfstartview={FitH},    
    pdfnewwindow=true,      
    colorlinks=true,        
    linkcolor=linkcolor,    
    citecolor=linkcolor,    
    filecolor=linkcolor,    
    urlcolor =linkcolor     
}

\def\B{\mathbb{B}}

\def\R{\mathbb{R}}

\def\N{\mathbb{N}}

\def\uni{\mathcal{U}} 

\def\diag#1{\mathrm{diag}\left[ #1 \right]}
\def\tr#1{\mathrm{tr}\left[ #1 \right]}
\def\abs#1{\left\vert #1 \right\vert}
\def\det#1{\mathrm{det}\left[ #1 \right]}
\def\set#1{\left\lbrace #1 \right\rbrace}

\def\norm#1{\left\Vert #1 \right\Vert}
\def\card#1{\left\vert #1 \right\vert}
\def\t{^\intercal}

\def\v#1{\bm{#1}}
\def\m#1{\bm{\mathrm{#1}}}

\def\eqdef{:=}

\def\ind{\mathds{1}}
\def\logit{\ell}

\newcommand{\prob}[1]{\mathbb{P}\left(#1 \right)}
\newcommand{\probplus}[2]{\mathbb{P}_{#1}\left(#2 \right)}

\newcommand{\evplus}[2]{\mathbb{E}_{#1}\left(#2 \right)}

\def\LinQu{{\scriptscriptstyle\text{LinQu}}}
\def\ExpQu{{\scriptscriptstyle\text{ExpQu}}}

\def\LogCo{{\scriptscriptstyle\text{LogCo}}}
\def\Prod{{\scriptscriptstyle\text{Prod}}}
\def\Gau{{\scriptscriptstyle\text{GauC}}}
\def\Poi{{\scriptscriptstyle\text{Poi}}}

\newtheorem{theorem}{Theorem}[section]

\newtheorem{proposition}[theorem]{Proposition}
\newtheorem{corollary}[theorem]{Corollary}

\newenvironment{remark}[1][Remark]{\begin{trivlist}
\item[\hskip \labelsep {\bfseries #1}]}{\end{trivlist}}

\def\keywords#1{\par\addvspace\medskipamount{\rightskip=0pt plus1cm
\def\and{\ifhmode\unskip\nobreak\fi\ $\cdot$
}\noindent \textbf{Keywords}\enspace\ignorespaces#1\par}}

\addtokomafont{footnote}{\sffamily\fontsize{6}{3}\selectfont}
\author{Christian Sch\"afer$^{1,2}$}
\def\crest{
\footnotetext[1]{Centre de Recherche en Économie et Statistique, 3 Avenue Pierre Larousse, 92240 Malakoff, France}
\footnotetext[2]{CEntre de REcherches en MAthématiques de la DEcision, Université Paris-Dauphine, Place du Maréchal de Lattre de Tassigny
75775 Paris, France}
}

\usepackage{algorithmic}
\usepackage{algorithm}

\floatname{algorithm}{Procedure}

\title{Parametric families for Monte Carlo on binary spaces}

\hypersetup{
    pdftitle={Parametric families for Monte Carlo on binary spaces},
    pdfkeywords={Binary parametric families}{Sampling correlated binary data}
}

\begin{document}

\maketitle
\crest
\thispagestyle{empty}

\begin{abstract}
In the context of adaptive Monte Carlo algorithms, we cannot directly generate independent samples from the distribution of interest but use a proxy which we need to be close to the target. Generally, such a proxy distribution is a parametric family on the sampling spaces of the target distribution. For continuous sampling problems in high dimensions, we often use the multivariate normal distribution as a proxy for we can easily parametrize it by its moments and quickly sample from it. The objective is to construct similarly flexible parametric families on binary sampling spaces too large for exhaustive enumeration.
\keywords{Binary parametric families \and Sampling correlated binary data}
\end{abstract}

\section{Introduction}
\subsection{Parametric families for Monte Carlo}
We discuss parametric families on binary spaces against the backdrop of Monte Carlo applications. The construction of binary parametric families $q_{\theta}$ that can model and reproduce the dependence structure of the target distribution $\pi$ is a difficult task, and many concepts of modeling multivariate binary data fail to provide parametric families that are suitable for adaptive Monte Carlo algorithms. Therefore, we do not only discuss workable families but also approaches that are impractical in order to provide a thorough review of all available methods.

\subsection{Notation}
We denote scalars in italic type, vectors in italic bold type and matrices in straight bold type. We write $\diag{\v a}$ for the diagonal matrix of the vector $\v a$ and $\diag{\m A}$ for the main diagonal of the matrix $\m A$. The determinant is denoted by $\det{\m A}$. We write $a_{i\bullet}$ and $a_{\bullet j}$ for the $i$th row and $j$th column of $\m A$, respectively. We write $\m A\succ0$ to indicate that $\m A$ is positive definite. Given as set $M$, we write $\card M$ for the number of its elements, $\overline M$ for its closure and $\ind_M$ for its indicator function.

We write $\B=\set{0,1}$ for the binary space and denote by $d\in\N$ the generic dimension. Given a vector $\v\gamma\in\B^d$ and an index set $I\subseteq\set{1,\dots,d}$, we write $\v\gamma_I\in\B^{\card I}$ for the sub-vector indexed by $I$ and $\v\gamma_{-I}\in\B^{d-\card I}$ for its complement. If $I$ is a sequence $\set{i,\dots,j}$ we use the more explicit notation $\v\gamma_{i:j}$ instead of $\v\gamma_I$ and $\v\gamma_i$ if $I=\set i$.

We write $\v\gamma_{I_1}$ and $\v\gamma_{I_0}$ for $\v\gamma$ with its components indexed by $I$ set to $\v 1$ and $\v 0$, respectively. In particular, we frequently use the short notation $\v a_{i\bullet}\t\v\gamma_{i_1}$ for $a_{ii}+\sum_{j=1}^{i-1} a_{ij}\gamma_j$ where $\m A$ is a lower triangular matrix.

\subsection{Data from the target distribution}
In the sequel, let $d>0$ denote the dimension of the binary space $\B^d=\set{0,1}^d$. Adaptive Monte Carlo algorithms are generally able to produce a, not necessarily independent and possibly weighted, sample
\begin{equation*}
\v w=(w_1,\dots,w_n)\in[0,1]^n,\quad \m X=(\v x_1,\dots,\v x_n)\t\in\B^{n\times d}
\end{equation*}
from the target distribution $\pi$ we want to emulate using a binary family. We define the index set $D=\set{1,\dots,d}$ and denote by
\begin{equation}
\label{eq:sample mean}
\textstyle\bar x_i\eqdef\sum_{k=1}^n w_k x_{k,i},\quad \bar x_{i,j}\eqdef\sum_{k=1}^n w_k x_{k,i}x_{k,j},\quad i,j\in D
\end{equation}
the weighted first and second sample moments. We further define by\vspace{-2mm}
\begin{equation}
\label{eq:sample corr}
r_{i,j}\eqdef\frac{\bar x_{i,j}-\bar x_i\bar x_j}{\sqrt{\bar x_i(1-\bar x_i)\bar x_j(1-\bar x_j)}},\qquad i,j\in D.
\end{equation}
the weighted sample correlation.

\subsection{Suitable parametric families}
\label{sec:properties}
We first frame some properties making a parametric family suitable as sampling distribution in adaptive Monte Carlo algorithms.
\begin{enumerate}[(a)]
\item For reasons of parsimony, we want to construct a family of distributions with at most $\mathrm{dim}(\theta)\leq d(d+1)/2$ parameters.
\item Given a sample $\m X=(\v x_1,\dots,\v x_n)\t$ from the target distribution $\pi$, we need to estimate $\theta^*$ such that the binary family $q_{\theta^*}$ is close to $\pi$.
\item We need to generate samples $\m Y=(\v y_1,\dots, \v y_m)\t$ from the family $q_\theta$. We need the rows of $\m Y$ to be independent.
\item For some algorithms, we need to evaluate the probability $q_\theta(\v y)$. For instance, we need $q_\theta(\v y)$ to compute importance weights or acceptance ratios in the context of Importance Sampling or Markov chain Monte Carlo, respectively.
\item Analogously to the multivariate normal, we need our calibrated binary family $q_{\theta^*}$ to reproduce the marginals and covariance structure of $\pi$.
\end{enumerate}

%


%

\section{Distributions on binary spaces}
Before we embark on the discussion of binary families, we make some observations which hold true for every binary distribution. The notation and results introduced in this section will be used throughout the rest of this work. Here, we denote by $\pi$ some generic distribution on $\B^d$

\paragraph{Moments}
We use the short notation,
\begin{align*}
\textstyle
u_I(\v\gamma)\eqdef \prod_{i\in I}\gamma_i,\qquad I\subseteq D,
\end{align*}
for the product of all components index by $I$ with $\prod_{i\in \emptyset}=1$. Since $u_I(\v\gamma)=1$ iff $\v\gamma_i=1$ for all $i\in I$, $u_I$ is the indicator function for the unit vector $\v 1_{\abs I}$. We can characterize every distribution on $\B^d$ by $2^d-1$ full probabilities
\begin{equation*}
p_I\eqdef\probplus{\pi}{\v\gamma_I=1,\v\gamma_{D\setminus I}=0},\qquad I\subseteq D
\end{equation*}
or by $2^d-1$ cross-moments, that is marginal probabilities,
\begin{equation*}
\textstyle
m_I
\eqdef\evplus{\pi}{u_I(\v\gamma)}=\probplus{\pi}{\v\gamma_I=\v 1},\qquad  I\subseteq D.
\end{equation*}
In the following, we assume that $m_i\in(0,1)$ for all $i\in D$, since for $m_i\in\set{0,1}$, the component $\gamma_i=m_i$ is constant and therefore not part of the sampling problem. 

For the product of components normalized to have zero mean and unit variance, we write
\begin{align*}
\textstyle
v_I(\v\gamma)\eqdef \prod_{i\in I}(\gamma_i-m_i)/\sqrt{m_i(1-m_i)},\qquad I\subseteq D.
\end{align*}
Note that $\evplus{\pi}{v_{i,j}}$ is the correlation between $\gamma_i$ and $\gamma_j$. Therefore, we call
\begin{equation*}
c_I\eqdef\evplus{\pi}{v_I(\v\gamma)}
\end{equation*}
the correlation of order $\abs I$.

\paragraph{Marginals}
We use the notation
\begin{equation*}
\pi_I(\v\gamma_I)=\textstyle\sum_{\v\xi\in\B^{d-\abs I}} \pi(\v\gamma_I, \v \xi), \qquad I\subseteq D.
\end{equation*}
for the marginal distributions. Note the connection to the cross-moments
\begin{equation}
\label{eq:marginals and moments}
\pi_I(\v 1_{\abs I})
\textstyle=\textstyle\sum_{\v\xi\in\B^{d-\abs I}} \pi(\v 1_{\abs I}, \v \xi)
\textstyle=\sum_{\v \gamma\in\B^d} u_I(\v\gamma)\ \pi(\v\gamma)=m_I.
\end{equation}

\paragraph{Representations}
Let $\pi$ be the mass function of a binary distribution and suppose there is a bijective mapping $\tau\colon \R \supseteq  V\to\pi(\B^d)$. There are coefficients $a_{I}\in\R$ such that
\begin{align}
\label{eq:representations}
\textstyle\pi(\v\gamma)=\tau\left[\sum_{I\subseteq D} a_{I} \prod_{i\in I}\gamma_i\right].
\end{align}
\begin{proof}
Immediate from the representation of the Dirac delta function as a product,
\begin{align*}
\textstyle\pi(\v\gamma)=\tau\left[\sum_{I\subseteq D}\delta_{\kappa^{I}}(\v \gamma)\tau^{-1}(\pi(\kappa^{I}))\right],\quad 
\delta_{\kappa^{I}}(\v \gamma)=\prod_{i\in I}\gamma_i\prod_{i\in \set{1,\dots,d}\setminus I}(1-\gamma_i),
\end{align*}
where $\kappa^{I}$ denotes the vector with $\kappa^{I}_i=\ind_I(i)$ for all $i\in\set{1,\dots,d}$.
\end{proof}

\paragraph{Constraints}
The general constraints on binary data are
\begin{equation}
\label{eq:bin bounds}
\textstyle
\left(\sum_{i\in I}m_i-\vert I\vert+1\right)\vee 0
\leq m_I
\leq \min\set{m_K\mid K\subseteq I},
\end{equation}
where the upper bound is the monotonicity of the measure, and the lower bound follows from
\begin{align*}
\abs I -1
&\textstyle =\sum_{\v\gamma\in\m\B^d}(\vert I\vert-1)\pi(\v\gamma) \\
&\textstyle \geq\sum_{\v\gamma\in\m\B^d}\left(\sum_{i\in I}\gamma_i-u_I(\gamma)\right)\pi(\v\gamma) \\
&\textstyle =\sum_{i\in I}m_i-m_I.
\end{align*}
In fact, $m_I$ is a $\abs I$-dimensional copula with respect to the expectations $m_i$ for $i\in I$, see \citet[p.45]{nelsen2006introduction}, and the inequalities (\ref{eq:bin bounds}) correspond to the Fr\'echet-Hoeffding bounds.

\paragraph{Sampling}
For sampling from a binary distribution $\pi$, we apply the chain rule factorization
\begin{equation}
\begin{split}
\label{eq:chain rule factorization}
\pi(\v\gamma)
&\textstyle =\pi_{\set 1}(\v\gamma_1)\prod_{i=2}^d \pi_{\set{1:i}}(\v\gamma_i\mid\v\gamma_{1:i-1}) \\
&\textstyle =\pi_{\set 1}(\v\gamma_1)\prod_{i=2}^d \pi_{\set{1:i-1}}(\v\gamma_{1:i-1})/\pi_{\set{1:i}}(\v\gamma_{1:i}),
\end{split}
\end{equation}
which permits to sample a random vector component-wise, conditioning on the entries we already generated. We do not even need to compute the full decomposition \eqref{eq:chain rule factorization}, but only the conditional probabilities $\pi_{\set{1:i}}(\gamma_i=1\mid\v\gamma_{1:i-1})$ defined by
\begin{equation}
\label{eq:cond probabilities}
\frac{\pi_{\set{1:i}}(\v\gamma_{1:i-1},1)}
{\pi_{\set{1:i}}(\v\gamma_{1:i-1},1)+\pi_{\set{1:i}}(\v\gamma_{1:i-1},0)}.
\end{equation}
The full probability $\pi(\v\gamma)$ is then computed as a by-product of Procedure \ref{algo:sampling}.

\begin{algorithm}[H]
\caption{Sampling via chain rule factorization}
\label{algo:sampling}
\begin{algorithmic}
\STATE $\v y=(0,\dots,0),\ p\gets 1$
\FOR {$i=1\dots,d$}\vspace{0.2em}
  \STATE $r\gets \pi_{\set{1:i}}(\gamma_i=1\mid\gamma_{1:i-1})$ \\[0.2em]
  \STATE \textbf{sample } $u\sim\uni_{[0,1]},\ y_i\gets\ind_{[0,r]}(u)$ \\[0.2em]
  \STATE $p\gets\begin{cases}
	    p\cdot r    & \textbf{if }\ \ y_i=1 \\
	    p\cdot (1-r) & \textbf{if }\ \ y_i=0
            \end{cases}$ \\[0.2em]
\ENDFOR
\RETURN $\v y,\ p$
\end{algorithmic}
\end{algorithm}

%


%

\section{The product family}
The simplest non-trivial distributions on $\B^d$ are certainly those having independent components.

\subsection{Definition}
For a vector $\m m\in(0,1)^d$ of marginal probabilities, we define the product family
\begin{equation}
\begin{split}
\label{eq:product family}
q_{\v m}^\Prod(\v \gamma)
&\textstyle \eqdef \prod_{i\in D} m_i^\gamma(1-m_i)^{1-\gamma} \\
&\textstyle \stackrel{\quad\ }{=} \prod_{i\in D} (1-m_i)\exp\left(\sum_{i\in D} \logit(m_i) \right).
\end{split}
\end{equation}
The second representation using the logit function
\begin{equation}
\label{eq:logit}
\logit\colon(0,1)\to\R,\quad \logit(p)=\log p- \log(1-p)
\end{equation}
is useful to identify the product family as special case of more complex families.

\subsection{Properties}
We check the requirement list from Section \ref{sec:properties}:
\begin{enumerate}[(a)]
\item The product family is parsimonious with $\mathrm{dim}(\theta)=d$.
\item The maximum likelihood estimator $\v m^*$ is the sample mean \eqref{eq:sample mean}.
\item We easily sample from $q_{\v m}^\Prod$, since \eqref{eq:chain rule factorization} holds trivially.
\item We easily evaluate the probability of a product of independent components.
\item \color{alarmcolor} The family $q_{\v m}^\Prod$ does not reproduce dependencies we might observe in the data $\m X$.
\end{enumerate}
The last point is a weakness which makes this simple family impractical when adaptive Monte Carlo algorithms are applied to challenging sampling problems. The product family $q_{\v m}^\Prod$ is might often fail to mimick the target distribution $\pi$ sufficiently well. Therefore, the rest of this paper deals with ideas on how to sample binary vectors with a given dependence structure. 

\subsection{Beyond the product family}
There are, to our knowledge, two main strategies to produce binary vectors with correlated components.
\begin{enumerate}[(1)]
\item We can construct a generalized linear family which permits computation of its marginal distributions. We apply the chain rule factorization \eqref{eq:chain rule factorization} and write $q_\theta$ as
\begin{equation}
\textstyle
q_{\theta}(\v\gamma)=q_{\theta}(\v\gamma_1)\prod_{i=2}^d q_{\theta}(\v\gamma_i\mid\v\gamma_{1:i-1}),
\end{equation}
which allows us to sample vectors component-wise.
\item We sample from a multivariate auxiliary distribution $h_{\theta}$ dichotomize the samples, that is map them into $\B^d$. We call
\begin{equation}
\textstyle
\label{eq:copula family}
q_{\theta}(\v\gamma)=\int_{\tau^{-1}(\v \gamma)} h_\theta(\v v) d\v v
\end{equation}
a copula family since we exploit the copula structure of the underlying distribution to build a new parametric family. However, we refrain from working with explicit uniform marginals which is not all necessary \citep{mikosch2006copulas}.
\end{enumerate}
In the following, we first study a few generalized linear families and then review a some copula approaches.

%


%
\section{The linear quadratic family}
Taking $\tau$ the identity mapping in \eqref{eq:representations}, we obtain a full linear representation
\begin{equation*}
\textstyle \pi(\v\gamma)=\sum_{I\subseteq D}a_I\ u_I(\v\gamma).
\end{equation*}
However, we cannot give a useful interpretation of the coefficients $a_I$. \citet{bahadur61representation} derived the following representation:
\begin{proposition}
\label{prop:bahadur}
We can write any binary distribution as
\begin{equation*}
\pi(\v\gamma)=q^\Prod_{\v m}(\v\gamma)\, (\textstyle\sum_{I\subseteq D} v_I(\v\gamma) \ c_I),
\end{equation*}
where $\v m=(m_1,\dots,m_d)$ are the marginal probabilities.
\end{proposition}
\begin{proof}
For convenience, we provide the proof of \citeauthor{bahadur61representation} in Appendix \ref{proof:bahadur}.
\end{proof}
This decomposition, first discovered by Lazarsfeld, is a special case of a more general interaction theory \citep{streitberg1990lancaster} and allows for a reasonable interpretation of the parameters. Indeed, we have a product family times a correction term $1+\sum_{I\in\mathcal{I}_k} v_I(\v\gamma) \ c_I$ where the coefficients are higher order correlations.

\subsection{Definition}
We can try to construct a more parsimonious family by removing higher order interaction terms. For additive approaches, however, we face the problem that a truncated representations do not necessarily define probability distributions since they might not be non-negative.

Still, for a symmetric matrix $\m A$, we define the $d(d+1)/2$ parameter family
\begin{equation}
\label{eq:linear family}
q_{\m A,a_0}^\LinQu(\v\gamma)=\mu(a_0 + \v\gamma\t\m A\v\gamma),
\end{equation}
where $\mu>0$ is a normalizing constant and we set $a_0=-(\min_{\v\gamma\in\B^d} \v\gamma\t\m A\v\gamma\wedge 0$). Since $a_0$ is the solution of an NP hard quadratic unconstrained binary optimization problem, this definition is of little practical value.

\subsection{Moments}
In virtue of the linear structure, we can derive explicit expressions for the cross-moments and marginal distributions, explicit meaning that the complexity is polynomial in $d$. The proofs are basic but rather tedious, so we moved them to the appendix section. 

Next, we give a general formula yielding all cross-moments, including the normalizing constant. 
\begin{proposition}
\label{prop:cross moments}
For a set of indices $I\subseteq D$, we can write the corresponding cross-moment as
\begin{equation*}
m_I=\frac{1}{2^{\abs I}}+
\frac{
\sum_{i\in I} \left\lbrack 2\sum_{j\in D} a_{i,j}+\sum_{j\in I\setminus\set i}^d a_{i,j} \right\rbrack
}{
2^{\abs I}(4a_0+\v 1\t \m A \v 1 + \tr{\m A})
}.
\end{equation*}
For a proof see Appendix \ref{proof:cross moments}
\end{proposition}

\begin{corollary}
The normalizing constant is
\begin{equation*}
\mu=2^{-d+2}\left(4a_0+\v 1\t \m A \v 1 + \tr{\m A}\right)^{-1},
\end{equation*}
and the expected value is
\begin{equation*}
\evplus{q_{\m A,a_0}^\LinQu}{\gamma_i}=\frac{1}{2}+\frac{\sum_{k=1}^d\, a_{i,k}}{4a_0+\v 1\t \m A \v 1 + \tr{\m A}}.
\end{equation*}
\end{corollary}

\noindent
The mean $m_i$ is close to $1/2$ unless the row $\v a_i$ dominates the matrix. Therefore, if $\m A$ is non-negative definite, the marginal probabilities $m_i$ can hardly take values at the extremes of the unit interval.

\subsection{Marginals}
For the marginal distributions
\begin{equation*}
q_{\m A, a_0}^{(1:k)}(\v\gamma_{1:k})=\textstyle\sum_{\v \xi\in\B^{d-(k+1)}}q_{\m A, a_0}(\v\gamma_{1:k},\v \xi)\
\end{equation*}
there are explicit and recursive formulas. Hence, we can compute the chain rule decomposition \eqref{eq:chain rule factorization} which in turn allows to sample from the family.
\begin{proposition}
\label{prop:marginals}
For the marginal distribution holds
\begin{align*}
q_{\m A, a_0}^{(1:k)}(\v\gamma_{1:k})=\mu 2^{d-k-2}s_k(\v\gamma_{1:k}),
\end{align*}
where
\begin{align*}
\textstyle
s_k(\v\gamma_{1:k})=4a_0
&\textstyle+\sum_{i=1}^{k}\gamma_i\,\left(\sum_{j=1}^{k}\gamma_j a_{i,j}+ \sum_{j=k+1}^{d} a_{i,j}\right) \\
&\textstyle+\sum_{i=k+1}^d \sum_{j=k+1}^d a_{i,j} + \sum_{i=k+1}^d a_{i,i}.
\end{align*}
For a proof see Appendix \ref{proof:marginals}
\end{proposition}
Recall the connection between marginal distributions and moments we observed in \eqref{eq:marginals and moments}. For $\v\gamma_I=\v 1$ we obtain
\begin{align*}
s_I(\v 1_k)
&=\textstyle4a_0+
4\sum_{i\in I}(\sum_{j\in I} a_{i,j}
+ \sum_{j\in I^c} a_{i,j})
\\
&\textstyle\qquad 
+ \sum_{i\in I^c} \sum_{j\in I^c} a_{i,j}
+ \sum_{i\in I^c} a_{i,i}\\
&=\textstyle4a_0+
\sum_{i\in D}\sum_{j\in D}a_{i,j}+\sum_{i\in D}a_{i,i}
+ 3\sum_{i\in I}\sum_{j\in I} a_{i,j}
\\
&\textstyle\qquad 
+ 2 \sum_{i\in I}\sum_{j\in I^c} a_{i,j}
- \sum_{i\in I} a_{i,i}\\
&=\textstyle4a_0+
\v 1\t \m A \v 1+\tr{\m A}+
\\
&\textstyle\qquad 
\sum_{i\in I}\left\lbrack2\sum_{j\in D} a_{i,j}+\sum_{j\in I\setminus \set i} a_{i,j}\right\rbrack,
\end{align*}
and $\pi_I(\v 1_k)=\mu 2^{d-\abs I -2}s_I(\v 1_k)$ is indeed the expression for the cross-moments in Proof of Proposition \ref{prop:cross moments}.

\subsection{Fitting the parameter}
\label{sec:fitting linear family}
Given a sample $\m X=(\v x_1,\dots,\v x_n)\t\sim\pi$ from the target distribution, we can determine $a_0$ and a matrix $\m A$ such that the family $q^\LinQu_{\m A, a_0}$ fits the first and second sampling moments
\begin{equation*}
\textstyle
\bar{x}_{\set{i,j}}=n^{-1}\sum_{k=1}^n x_{k,i}x_{k,j},\quad i,j\in D
\end{equation*}
by solving a linear system of dimension $d(d+1)/2+1$. We first use the bijection
\begin{equation*}
\tau\colon D\times D\to \set{1,\dots,d(d+1)/2},\quad
\tau(i,j)=i(i-1)/2+j 
\end{equation*}
to map symmetric matrices into $\mathbb R^{(d+1)\,d/2}$. Precisely, for the matrices $\m A$ and $\overline{\m X}$, we define the vectors
\begin{equation*}
\hat a_{\tau(i,j)}
\eqdef a_{i,j},\quad 
\hat x_{\tau(i,j)}
\eqdef \bar x_{i,j}
\end{equation*}
and the design matrix
\begin{equation*}
\hat s_{\tau(i,j),\tau(k,l)}
\eqdef 2^{\ind_{\set{i,j}}(k)+\ind_{\set{i,j,k}}(l)}.
\end{equation*}
Note that $\abs{\hat{\v a}}=\v 1\t \m A \v 1+\tr{\m A}$. We then equate the distribution moments to the sample moments and normalize such that
\begin{align}
2^{d-2}(\m I\,a_0+\frac{1}{4}\, \hat{\m S}\hat{\v a})=\hat{\v x},\quad 2^{d-2}(4a_0+\abs{\hat{\v a}})=1.
\end{align}
The solution of the linear system
\begin{equation*}
\begin{pmatrix}
\hat{\v a}^* \\
a_0^*
\end{pmatrix}
=
2^{-d+2}
\begin{bmatrix}
\frac{1}{4}\, \hat{\m S} & \v 1 \\
\,4\, \v 1 \t    &    1 \\
\end{bmatrix}^{-1}
\begin{pmatrix}
\hat{\v x} \\
1
\end{pmatrix}
\end{equation*}
is finally transformed back into a symmetric matrix $\m A^*$. Since the design matrix does not depend on the data, fitting several parameters to different data on the same space $\B^d$ is extremely fast.

\subsection{Properties}
We check the requirement list from Section \ref{sec:properties}:
\begin{enumerate}[(a)]
\item The linear family is sufficiently parsimonious having dimension $\mathrm{dim}(\theta)=d(d+1)/2$.
\item {\color{alarmcolor}We can fit the parameters $\m A$ and $a_0$ via method of moments. However, the fitted function $q_{\m A^*,a_0^*}^\LinQu(\v \gamma)$ is usually not a distribution.}
\item We can sample via chain rule factorization.
\item We can evaluate $q_{\m A,a_0}^\LinQu(\v y)$ via chain rule factorization while sampling.
\item The family $q_{\m A,a_0}^\LinQu$ reproduces the mean and correlations of the data $\m X$.
\end{enumerate}
Since in applications, the fitted matrix $\m A^*$ is hardly ever positive definite, we cannot use the linear family in an adaptive Monte Carlo context. As other authors \citep{park1996simple,emrich1991method} remark, additive representations like Proposition \ref{prop:bahadur} are instructive but we cannot derive practical families from them.

%


%

\section{The exponential quadratic family}
If $\pi(\v\gamma)>0$ for all $\v\gamma\in\B^d$, we can use $\tau=\exp$ in \eqref{eq:representations} and obtain a full log-linear representation
\begin{equation*}
\textstyle
\pi(\v\gamma)=\exp\left(\sum_{I\subseteq D} a_I\ u_I(\v\gamma)\right).
\end{equation*}

Note that we assume the probability mass function $\pi$ is assumed to be log-linear in the parameters $a_I$. In the context of contingency tables the term ``log-linear family`` refers to the assumption that the marginal probabilities $m_I$ are log-linear in the higher order marginals.

\begin{remark}
Contingency table analysis is a well studied approach to modeling discrete data \citep{bishop75discrete,christensen1997log}. For binary data, the underlying sampling distribution is assumed to be multinomial which requires an enumeration of the state space we want to avoid. \citet{gange1995generating} uses the Iterative Proportional Fitting algorithm \citep{haberman1972algorithm} from log-linear interaction theory to construct a binary distribution with given marginal probabilities. The fitting procedures, however, require storage of all configurations $\pi_I(\v\gamma_I)$ and the construction of the joint posterior from the fitted marginal probabilities. The method is powerful and exact but computationally infeasible even for moderate dimensions.
\end{remark}

\subsection{Definition}
Removing higher order interaction terms, we can construct a $d(d+1)/2$ parameter family
\begin{equation}
\label{eq:eqm}
q_{\m A}^\ExpQu(\v\gamma)\eqdef \mu\exp(\v\gamma\t\m A\v\gamma),
\end{equation}
where $\m A$ is a symmetric matrix and $\mu\eqdef[\sum_{\v\gamma\in\B^d}\exp(\v\gamma\t\m A\v\gamma)]^{-1}$. We recognize the product family \eqref{eq:product family} as the special case $\mu=\prod_{i\in D}(1-m_i)^d$ and $\m A=\diag{\logit(\v m)}$.

\subsection{Marginals}
The moments or marginal distributions of $q_{\m A}$ are sums of exponentials which, in general, do not simplify to expressions that are polynomial in $d$. Therefore, we cannot perform a chain rule factorization \eqref{eq:chain rule factorization} to sample from the family.

\citet{cox1994note} proposed the following second degree Taylor approximations to the marginal distributions which are again of the form \eqref{eq:eqm}.
\begin{proposition}
\label{prop:approx marginals}
We write the parameter $\m A$ as
\begin{equation}
\label{eq:matrix representation}
\m A=\begin{pmatrix}
	\m A' & \v b\t \\
	\v b  & c
  \end{pmatrix}, 
\end{equation}
and define the parameters
\begin{align*}
\tilde{\m A}_{d-1}
&=\textstyle
\m A'
+\left(1+\tanh(\frac{c}{2})\right)\diag{\v b}+\frac{1}{2}\,\mathrm{sech}^2(\frac{c}{2})\v b \v b\t, \\
\tilde\mu_{d-1}
&=
\mu(1+\exp(c))
\end{align*}
Then $q_{\tilde{\m A}_{d-1}}(\v\gamma_{\,1:d-1})$ is the second degree Taylor approximation to the marginal distribution $q_{\m A_{1:d-1}}(\v\gamma_{\,1:d-1})$. For a proof see Appendix \ref{proof:approx marginals}.
\end{proposition}

\noindent
If we recursively compute $q_{\tilde{\m A}_{d-1}},\dots,q_{\tilde{\m A}_1}$, we can derive approximate conditional probabilities using \eqref{eq:cond probabilities}. Precisely, we have
\begin{equation}
\label{eq:log-linear approximate marginals}
q_{\tilde{\m A}_{i}}(\gamma_i=1\mid\v \gamma_{1:i-1})=\logit^{-1}(\tilde c_i + \tilde{\v b}_i\t\gamma_{1:i-1}),
\end{equation}
where $\logit^{-1}(x)=(1+\exp(-x))^{-1}$ and $\tilde c_i,\ \tilde{\v b}_i$ are parts of the matrix $\tilde{\m A}_i$ according to the notation introduced in \eqref{eq:matrix representation}. In particular, \eqref{eq:log-linear approximate marginals} is a logistic regression. We come back to this class of families in the following Section \ref{sec:logistic conditionals family}. We can sample from the proxy
\begin{equation*}
\textstyle
\tilde q_{\tilde{\m A}}(\v\gamma)\eqdef\prod_{i\in D} q_{\tilde{\m A}_i}(\gamma_i\mid\v \gamma_{1:i-1})\approx q_{\m A}^\ExpQu(\v\gamma),
\end{equation*}
which is close to the original exponential quadratic family. The goodness of the approximation might be improved by judicious permutation of the components. The approximation error is hard to control, however, since we repeatedly apply the second degree approximation and propagate initial errors. 

\subsection{Fitting the parameter}
As in section \ref{sec:fitting linear family}, we use the bijection
\begin{equation*}
\tau\colon D\times D\to \set{1,\dots,d(d+1)/2},\quad
\tau(i,j)=i(i-1)/2+j 
\end{equation*}
to map symmetric matrices into $\R^{(d+1)\,d/2}$. Precisely, for the matrices $\m A$ and $\overline{\m X}$, we define the vectors
\begin{equation*}
\hat a_{\tau(i,j)}
\eqdef a_{i,j},\quad 
\hat x_{\tau(i,j)}
\eqdef \bar x_{i,j}.
\end{equation*}
We let $y_k=\log \pi(\v x_k)$ for $k=1,\dots,n$ and fit the family solving the least square problem
\begin{equation*}
\textstyle\min_{\hat{\v a}\in\R^{(d+1)\,d/2}} \norm{\hat{\m X} \hat{\v a}-\v y}_2
\end{equation*}
which yields the parameters
\begin{equation*}
a_{i,j}^*=\lbrack(\hat{\m X}\t\hat{\m X})^{-1}\hat{\m X}\t\v y \rbrack_{\tau(i,j)}.
\end{equation*}
Note that in most adaptive Monte Carlo algorithms that involve importance sampling or Markov transitions, the probabilities $\pi(\v x_k)$ of the target distribution are already computed such that the fitting procedure is rather fast.

\subsection{Properties}
We check the requirement list from Section \ref{sec:properties}:
\begin{enumerate}[(a)]
\item The log-linear family is sufficiently parsimonious with $\mathrm{dim}(\theta)=d(d+1)/2$.
\item We can fit the parameter $\m A$ via minimum least squares.
\item {\color{alarmcolor} We can sample from an approximation $\tilde q_{\tilde{\m A}}(\v\gamma)\approx q_{\m A}(\v\gamma)$ to the log-linear family. However, we cannot control the approximation error.}
\item We can evaluate $q_{\m A,a_0}(\v y)$ up to the normalization constant $\v \mu$ which suffices for most adaptive Monte Carlo methods.
\item The family $q_{\m A,a_0}$ reproduces the mean and correlations of the data $\m X$.
\end{enumerate}

%


%

\section{The logistic conditionals family}
\label{sec:logistic conditionals family}
In the previous section we saw that even for a rather simple non-linear family we cannot derive closed-form expressions for the marginal probabilities. Therefore, instead of computing the marginals for a $d$-dimensional family $q_\theta(\v\gamma)$, we directly fit univariate families
\begin{equation*}
q_{\theta}(\gamma_i=1\mid\gamma_{1:i-1}),\qquad i\in D
\end{equation*}
to the conditional probabilities $\pi(\gamma_i=1\mid\gamma_{1:i-1})$ of the target function. Precisely, we postulate the logistic relation
\begin{equation*}
\textstyle
\logit(\probplus{\pi}{\gamma_i=1})= b_{i,i}+\sum_{j=1}^{i-1} b_{i,j} \gamma_j,\quad i\in D
\end{equation*}
for the marginal probabilities where $\logit$ is the logit function defined in \eqref{eq:logit}.

\subsection{Definition}
For a $d$-dimensional lower triangular matrix $\m B$, we define the logistic conditionals family as
\begin{align}
\label{eq:lb}
\textstyle
q_{\m B}^\LogCo(\v \gamma)
&\eqdef \prod_{i\in D}
q^\Prod_{p(b_{i,i}+\v b_{i,1:i-1}\t \v \gamma_{1:i-1})}(\gamma_i) \\
&\nonumber=\exp\left[\ \sum_{i\in D}\left[
\gamma_i(b_{i,i}+\v b_{i,1:i-1}\t \v \gamma_{1:i-1})-\log\left(1+\exp(b_{i,i}+\v b_{i,1:i-1}\t \v \gamma_{1:i-1})\right)\right]\right]
\end{align}
where $q^\Prod_p$ is the Bernoulli distribution and
\begin{equation*}
p(x)=\logit^{-1}(x)=(1+\exp(-x))^{-1}
\end{equation*}
the logistic function. We immediately identify the product family $q^\Prod_{\v m}$ as the special case $\m B=\diag{\logit(\v m)}$. The logistic conditionals family is not in the exponential family.

Note that there are $d!$ possible logistic families and we arbitrarily pick one while there should be a permutation $\sigma(D)$ of the components which is optimal in a sense of nearness to the data. In practice, however, changing the parametrization does not seem to have a noticeably impact on the quality of the adaptive Monte Carlo algorithm.

\subsection{Sparse logistic regressions}
The major drawback of all multiplicative families is the fact that they do not have closed-form likelihood-maximizers such that the parameter estimation requires costly iterative fitting procedures. Therefore, we construct a sparse version of the logistic regression family which we can estimate faster than the saturated family.

Instead of fitting the parameter of the saturated family $q^\LogCo_{\v b}(\gamma_i\mid\gamma_{1:i-1})$, we preferably work with a more parsimonious regression family like $q^\LogCo_{\v b}(\gamma_i\mid\gamma_{L_i})$ for some index set $L_i\subseteq \lbrace1,\dots,i-1\rbrace$, where the number of predictors $\#L_i$ is typically smaller than $i-1$.

We solve this nested variable selection problem using some simple, fast to compute criterion. For $\varepsilon$ about $\frac{1}{100}$, we define the index set
\begin{equation*}
I\eqdef\lbrace i=1,\dots,d \mid \ \bar{x}_i\, \notin\, (\varepsilon,1-\varepsilon)\, \rbrace.
\end{equation*}
which identifies the components which have, according to the data, a marginal probability close to either boundary of the unit interval.

We do not fit a logistic regression for the components $i\in I$. We rather set $L_i=\emptyset$ and draw them independently, that is we set $b_{i,i}=\logit(\bar x_i)$ and $\v b_{i,-i}=\v 0$ which corresponds to logistic conditionals family without predictors. The reason is twofold. Firstly, interactions do not really matter if the marginal probability is excessively small or large. Secondly, these components are prone to cause complete separation in the data or might even be constant.

For the conditional distribution of the remaining components $I^c=D\setminus I$, we construct parsimonious logistic regressions. For $\delta$ about $\frac{1}{10}$, we define the predictor sets
\begin{equation*}
L_i\eqdef\lbrace j=1,\dots,i-1 \mid \delta < \abs{r_{i,j}} \rbrace,\quad i\in I^c,
\end{equation*}
which identifies the components with index smaller than $i$ and significant mutual association.

\subsection{Fitting the parameter}
Given a sample $\m X=(\v x_1,\dots,\v x_n)\t\sim\pi$ from the target distribution we regress $\v y^{(i)}=\m X_i$ on the columns $\m Z^{(i)}=(\m X_{1:i-1}, \v 1)$, where the column $\v Z_i^{(i)}$ yields the intercept to complete the logistic conditionals family.

We maximize the log-likelihood function $\ell(\v b)=\ell(\v b\mid\v y, \m Z)$ of a weighted logistic regression family by solving the first order condition $\partial\ell/\partial \v \beta=\v 0$. We find a numerical solution via Newton-Raphson iterations
\begin{equation}
\label{newton log regression}
-\frac{\partial^2\ell(\v b^{(r)})}{\partial \v b \v b\t}(\v
b^{(r+1)}-\v b^{(r)})
=\frac{\partial\ell(\v b^{(r)})}{\partial \v b}, \quad r>0,
\end{equation}
starting at some $\v b^{(0)}$; see Procedure \ref{algo:fit logistic} for the exact terms. Other updating formulas like Iteratively Reweighted Least Squares or quasi-Newton iterations should work as well.

\begin{algorithm}[H]
\caption{Fitting the weighted logistic regressions}
\label{algo:fit logistic}
\begin{algorithmic}
\REQUIRE{$\v w=(w_1,\dots,w_n),\ \m X=(x_1,\dots,x_n)\t,\ \m B\in\R^{d\times d}$\\[.5em]}
\FOR {$i\in I^c$}
  \STATE $\m Z\gets(\m X_{L_i},\v 1),\ \v y\gets\m X_i, \ \v b^{(0)}\gets\m B_{i,L_i\cup\lbrace i \rbrace}$ 
  \REPEAT
  \STATE
    \begin{tabular}{rllr}
    $p_k$       &\hspace{-3mm}$\gets$&\hspace{-3mm}$\logit^{-1}(\m Z_k \v b^{(r-1)})$&\textbf{ for all }$k=1,\dots,n$ \\[.5em]
    $q_k$       &\hspace{-3mm}$\gets$&\hspace{-3mm}$p_k(1-p_k)$                      &\textbf{ for all }$k=1,\dots,n$ \\[.5em]
    \end{tabular} \\
    \begin{tabular}{ll}
    $\v b^{(r)}\gets$&\hspace{-3mm}$\left(\m Z\t \diag{\v w} \diag{\v q} \m Z+\varepsilon\m I_n\right)^{-1} \times$ \\[.5em]
    \hfill           &\hspace{-3mm}$\left(\m Z\t \diag{\v w}\right)\left(\diag{\v q} \m Z\, \v b^{(r-1)}+\left(\v y - \v p\right)\right)$\\[.5em]
    \end{tabular}
  \UNTIL {$|b_j^{(r)}-b_j^{(r-1)}|<10^{-3}$ for all $j$}\\[.5em]
  \STATE $\m B_{i,L_i\cup\lbrace i \rbrace}\gets\v b$\\[.5em]
\ENDFOR
\RETURN $\m B$
\end{algorithmic}
\end{algorithm}

Sometimes, the Newton-Raphson iterations do not converge because the likelihood function is monotone and thus has no finite maximizer. This problem is caused by data with complete or quasi-complete separation in the sample points \citep{albert_84}. There are several ways to handle this issue.
\begin{enumerate}[(a)]
\item
We just halt the algorithm after a fixed number of iterations and ignore the lack of convergence. Such proceeding, however, might cause uncontrolled numerical problems.
\item
\citet{firth_93} recommends the Jeffreys prior for its bias reduction but this option is computationally rather expensive. We might instead use a Gaussian prior with variance $1/\varepsilon>0$ which adds a quadratic penalty term $\varepsilon\v b\t \v b$ to the log-likelihood to ensure the target-function is convex.
\item
As we notice that some terms of $\v b_i$ are growing beyond a certain threshold, we move the component $i$ from the set of components with associated logistic regression family $I^c$ to the set of independent components $I$.
\end{enumerate}
In practice, we recommend to combine the approaches (c) and (d). In Procedure \ref{algo:fit logistic}, we did not elaborate how to handle non-convergence, but added a penalty term to the log-likelihood, which causes the extra $\varepsilon\m I_n$ in the Newton-Raphson update. Since we solve the update equation via Cholesky factorizations, adding a small term on the diagonal ensures that the matrix is indeed numerically decomposable.

\subsection{Properties}
We check the requirement list from Section \ref{sec:properties}:
\begin{enumerate}[(a)]
\item The logistic regression family is sufficiently parsimonious with $\mathrm{dim}(\theta)=d(d+1)/2$.
\item We can fit the parameters $\v b_i$ via likelihood maximization for all $i\in D$. The fitting is computationally intensive but feasible.
\item We can sample $\v y\sim q^\LogCo_{\m B}$ via chain rule factorization.
\item We can exactly evaluate $q^\LogCo_{\m B}(\v y)$.
\item The family $q^\LogCo_{\m B}$ reproduces the dependency structure of the data $\m X$ although we cannot explicitly compute the marginal probabilities.
\end{enumerate}

\section{The Gaussian copula family}
In the preceding sections, we discussed three approaches based on generalized linear families. Now we turn to the second class of families we call copula families.

Let $h_\theta$ be a family of auxiliary distributions on $\mathcal X$ and $\tau\colon\mathcal X \to \B^d$ a mapping into the binary state space. We can sample from the copula family
\begin{equation*}
\textstyle q_{\theta}^{h,\tau}(\v\gamma)=\int_{\tau^{-1}(\v\gamma)}\,h_\theta(\v v)\,d\v v
\end{equation*}
by setting $\v y=h(\v v)$ for a draw $\v v\sim h_\theta$ from the auxiliary distribution.

\subsection{Definition}
Apparently, non-normal parametric distributions $s_{\theta}$ with at most $d(d-1)/2$ dependence parameters either have a very limited dependence structure or rather unfavorable properties \citep{joe1996families}. Therefore, the multivariate Gaussian distribution with
\begin{equation*}
h_{\m \Sigma}(\v v)=(2\pi)^{-\frac{d}{2}}\abs{\m \Sigma}^{-\frac{1}{2}}\exp(-\frac{1}{2}\,\v v\t\m \Sigma^{-1}\v v),
\end{equation*}
and mapping $\tau\colon\R^d\to\B^d$
\begin{equation*}
\tau_{\v \mu}(\v v)=(\ind_{(\infty,\mu_i]}(v_1),\dots,\ind_{(\infty,\mu_d]}(v_d)),
\end{equation*}
appears to be the natural and almost the only option for $h_\theta$. The Gaussian copula family, denoted by $q_{\v\mu,\m\Sigma}^\Gau$, has already been discussed repeatedly in the literature \citep{emrich1991method, leisch1998generation, cox2002some}.

\subsection{Moments}
For $I\subseteq D$, the cross-moment or marginal probabilities is
\begin{align*}
\textstyle
m_I
&\textstyle =\sum_{\v\gamma\in\B^d} q_{\v\mu,\m \Sigma}(\v1_I,\v\gamma_{D\setminus I})
\textstyle =\int_{\cup_{\v\gamma\in\B^d}\,\set{\tau_{\v \mu}^{-1}(\v 1_I,\v\gamma_{D\setminus I})}} \,h_{\m \Sigma}(\v v)\,d\v v \\
&\textstyle =\int_{\times_{i\in I} \set{\tau_{\mu_i}^{-1}(1)}} \,h_{\m \Sigma}(\v v)\,d\v v
\textstyle =\int_{\times_{i\in I}(-\infty,\mu_i]}\,h_{\m \Sigma}(\v v)\,d\v v,
\end{align*}
where we used \eqref{eq:marginals and moments} in the first line. Thus, the first and second moment of $q_{(\v \mu,\m \Sigma)}$ are
\begin{equation*}
m_i=\varPhi_1(\mu_i),\quad m_{i,j}=\varPhi_2(\mu_i,\mu_j;\sigma_{i,j})
\end{equation*}
where $\varPhi_1(v_i)$ and $\varPhi_2(v_i,v_j;\sigma_{i,j})$ denote the cumulative distribution functions of the univariate and bivariate normal distributions with zero mean, unit variance and correlation coefficient $\sigma_{i,j}\in[-1,1]$.

\subsection{Sparse Gaussian copulas}
We can speed up the parameter estimation and improve the condition of $\m \Sigma$, if we work with a parsimonious Gaussian copula. We can apply the same criterion we already introduced for the sparse logistic regression family. For $\varepsilon$ about $\frac{1}{100}$, we define the index set
\begin{equation*}
I\eqdef\lbrace i=1,\dots,d \mid \ \bar{x}_i\, \notin\, (\varepsilon,1-\varepsilon)\, \rbrace.
\end{equation*}
which identifies the components which have a marginal probability close to either boundary of the unit interval.

We do not fit a any correlation parameters for the components in $I$ but set $\sigma_{i,j}=0$ for all $j\in D\setminus \set i$. Firstly, the correlation does not really matter if the marginal probability is excessively small or large. Secondly, we fit the parameter $\m \Sigma$ by separately adjusting the bivariate correlations $\sigma_{i,j}$, and components with high correlations and extreme marginal probability lower the chance that $\m \Sigma$ is positive definite.

For the remaining components $I^c=D\setminus I$, we construct parsimonious Gaussian copula. For $\delta$ about $\frac{1}{10}$, we define the association set
\begin{equation*}
A\eqdef\set{\set{i,j}\in I^c\times I^c \mid \delta < \abs{r_{i,j}},\, i\neq j}
\end{equation*}
which identifies the components with significant correlation. For $i,j\in D\times D\setminus L$ we also set $\sigma_{i,j}=0$ to accelerate the estimation procedure.

\subsection{Fitting the parameter}
We fit the family $q\Gau_{(\v\mu,\m \Sigma)}$ to the data by adjusting $\v \mu$ and $\m \Sigma$ to the sample moments. Precisely, we solve the equations
\begin{align}
\label{eq:gaussian copula mean}
\varPhi_1(\mu_i)&=\bar x_i,\hspace{-2cm}& i\in D \\
\label{eq:gaussian copula corr}
\varPhi_2(\mu_i,\mu_j;\sigma_{i,j})&=\bar x_{i,j},\hspace{-2cm}& (i,j)\in A
\end{align}
with sample mean $\bar x_i$ and $\bar x_{i,j}$ as defined in \eqref{eq:sample mean}. We easily solve \eqref{eq:gaussian copula mean} by setting
\begin{equation*}
\mu_i=\Phi_1^{-1}(\bar x_i),\qquad i\in D.
\end{equation*}
The difficult task is computing a feasible correlation matrix from \eqref{eq:gaussian copula corr}. Recall the standard result \citep[p.255]{johnson2002continuous}
\begin{equation}
\label{eq:der of norm pdf}
\frac{\partial \Phi_2(y_1,y_2;\sigma)}{\partial\sigma}=h_{\sigma}(y_1,y_2),
\end{equation}
where $h_{\sigma}$ denotes the density of the bivariate normal distribution. We obtain the following Newton-Raphson iteration
\begin{equation}
\label{newton local parameters}
\alpha_{r+1}=\alpha_{r}-
\frac{\Phi_2(\mu_i,\mu_j;\alpha_r)-\bar x_{i,j}}{h_{\alpha_r}(\mu_i,\mu_j)},\quad (i,j)\in A,
\end{equation}
starting at some $\alpha_{0}\in(-1, 1)$. We use a fast series approximation \citep{drezner_98,divgi_79} to evaluate $\Phi_2(\mu_i,\mu_j;\alpha)$. These approximations are critical when $\alpha_r$ comes very close to either boundary of $[-1,1]$. The Newton iteration might repeatedly fail when restarted at the corresponding boundary $r_0\in\set{-1,1}$. This is yet another reason why it is preferable to work with a sparse Gaussian copula. In any event, $\Phi_2(y_1,y_2;\sigma)$ is monotonic in $\sigma$ since \eqref{eq:der of norm pdf}, and we can switch to bi-sectional search if necessary.

\begin{algorithm}[H]
\caption{Fitting the dependency matrix}
\label{algo:fit gaussian}
\begin{algorithmic}
\REQUIRE{$\bar x_i,\ \bar x_{i,j}$\textbf{ for all }$i,j\in D$\\[.5em]}
\STATE $\mu_i=\varPhi_{-1}(\bar x_i)$\textbf{ for all }$i\in D$
\STATE $\m \Sigma=\m I_d$
\FOR {$(i,j)\in A$}
  \REPEAT 
  \STATE $\displaystyle \sigma_{i,j}^{(r+1)}\gets
  \sigma_{i,j}^{(r)}-\frac{\Phi_2(\mu_i,\mu_j;\sigma_{i,j}^{(r)})-\bar x_{i,j}}{h_{\sigma_{i,j}^{(r)}}(\mu_i,\mu_j)}$\\[.5em]
  \UNTIL {$|\sigma_{i,j}^{(r)}-\sigma_{i,j}^{(r-1)}|<10^{-3}$}\\[.3em]
\ENDFOR
\STATE \textbf{if not }$\m\Sigma\succ 0$\textbf{ then }$\m \Sigma\gets(\m \Sigma+\abs{\lambda}\m I_d)/(1+\abs{\lambda})$
\RETURN $\v \mu,\, \m \Sigma$
\end{algorithmic}
\end{algorithm}

A rather discouraging shortcoming of the Gaussian copula family is that locally fitted correlation matrices $\m \Sigma$ might not be positive definite for $d \geq 3$. This is due to the fact that an elliptical copula, like the Gaussian, can only attain the bounds \eqref{eq:bin bounds} for $d<3$, but not for higher dimensions.

We propose two ideas to obtain an approximate, but feasible parameter:
\begin{enumerate}[(1)]
\item
\space We replace $\m \Sigma$ by $\m \Sigma^*=(\m \Sigma+\abs{\lambda}\m I)/(1+\abs{\lambda})$, where $\lambda$ is the smallest eigenvalue of the dependency matrix $\m \Sigma$. This approach evenly lowers the local correlations to a feasible level and is easy to implement on standard software. Alas, we make an effort to estimate $d(d-1)/2$ dependency parameters, and in the end we might not get more than an product family.

\item
We can compute the correlation matrix $\m \Sigma^*$ which minimizes the distance $\norm{\m \Sigma^*-\m \Sigma}_F$, where $\norm{\m A}^2_F=\tr{\m A \m A\t}$. In other words, we construct the projection of $\m \Sigma$ into the set of correlation matrices. \citet{higham_02} proposes an Alternating Projections algorithm to solve nearest-correlation matrix problems. Yet, if $\m \Sigma$ is rather far from the set of correlation matrices, computing the projection is expensive and, according to our experience, leads to troublesome distortions in the correlation structure.
\end{enumerate}

\subsection{Properties}
We check the requirement list from Section \ref{sec:properties}:
\begin{enumerate}[(a)]
\item The Gaussian copula family is sufficiently parsimonious with $\mathrm{dim}(\theta)=d(d+1)/2$.
\item We can fit the parameters $\v \mu$ and $\m \Sigma$ via method of moments. {\color{alarmcolor} The parameter $\m \Sigma$ is not always be positive definite which might require additional effort it feasible.}
\item We can sample $\v y\sim q^\Gau_{(\v \mu,\m \Sigma)}$ using $\v y=\tau_{\mu}(\v v)$ with $\v v\sim h_{\m \Sigma}$.
\item {\color{alarmcolor}We cannot evaluate $q^\Gau_{\m B}(\v y)$ since this requires computing a high-dimensional integral expression.}
\item The family $q^\Gau_{(\v \mu,\m \Sigma)}$ reproduces the mean and correlation structure of the data $\m X$.
\end{enumerate}
Obviously, we cannot use the Gaussian copula family in the context of importance sampling or Markov chain Monte Carlo, since evaluation of $q^\Gau_{(\v\mu,\m \Sigma)}(\v y)$ is not feasible. This family might be useful, however, in other adaptive Monte Carlo algorithms, for instance the Cross-Entropy method \citep{Rub:CE1} for combinatorial optimization.

\section{The Poisson reduction family}
Let $N=\set{1,\dots,n}$ denote another index set with $n\gg d$. Approaches to generating binary vectors that do not rely on the chain rule factorization \eqref{eq:chain rule factorization} are usually based on combinations of independent random variables
\begin{equation*}
\v v=(v_1,\dots,v_n)\sim \otimes_{k\in N} h_{\theta_k}.
\end{equation*}
We define index sets $\mathcal{M}=\set{S_i\in N \mid i\in D}$ and generate the entry $y_i$ via
\begin{equation*}
\tau_i\colon\mathcal{X}^{\abs{S_i}}\to\set{0,1},\quad \tau_i(\v v)=f(\v v_{S_i}),\quad i\in D.
\end{equation*}
In the context of Gaussian copulas, the auxiliary distributions $h_{\theta_k}=h_\theta$ are $d$ independent standard normal variables. \citet{park1996simple} propose the following family based on sums of independent Poisson variables.

\subsection{Definition}
We define a Poisson family $q^\Poi_{(\mathcal S,\lambda)}$ with auxiliary distribution
\begin{equation*}
\textstyle
h_{\lambda}(\v v)=\prod_{k\in N} (\lambda_k^{v_k}e^{-\lambda_k})/v_k!
\end{equation*}
and mapping $\tau\colon\N_0^{n}\to\B^d$
\begin{equation*}
\textstyle
\tau_{\mathcal S}(\v v)=(\ind_{\set 0}(\sum_{k\in S_1} v_k),\dots,\ind_{\set 0}(\sum_{k\in S_d} v_k)).
\end{equation*}

\subsection{Moments}
For an index set $I\in D$, the cross-moments or marginal probabilities are
\begin{equation*}
\textstyle
m_{I}=\prob{\forall i\in I\colon\sum_{k\in S_i}v_k=0}=\exp(-\sum_{k\in\cap_{i\in I}S_i} \lambda_k).
\end{equation*}
Therefore, fitting via method of moments is possible.

\begin{proposition}
\label{prop:poisson family}
For $\v\gamma\in\B^d$, define the index sets
\begin{equation*}
D_0=\set{i\in D\mid \gamma_i=0}, \quad 
D_1=\set{i\in D\mid \gamma_i=1},
\end{equation*}
and the families of subsets $\mathcal I_t=\set{I\in D_1\mid \abs I=t}$. We can write the mass function of the Poisson family as
\begin{align*}
q^\Poi_{(\mathcal S,\lambda)}(\v\gamma)
&=\textstyle \sum_{\v v \in \tau^{-1}(\v\gamma)}h_{\lambda}(\v v) \\
&=\textstyle m_{D_0}\bigg\lbrack1- \sum_{t=1}^{\abs{D_0}}(-1)^{t-1} \sum_{I\subseteq \mathcal I_t}\exp(-\sum_{k\in \cap_{i\in I} S_i\setminus \cup_{j\in D_1}S_j} \lambda_k) \bigg\rbrack.
\end{align*}
For a proof see Appendix \ref{proof:poisson family}.
\end{proposition}

\subsection{Fitting the parameter}
We need to determine the family of index sets $\mathcal{M}$ and the Poisson parameters $\lambda=(\lambda_1,\dots,\lambda_n)$ such that the resulting family $q_{(\mathcal{S},\lambda)}$ is optimal in terms of distance to the mean and correlation. Obviously, we face a rather difficult combinatorial problem. \citet{park1996simple} describe a greedy algorithm, based on convolutions of Poisson variables, that finds at least some feasible combination of $\mathcal S$ and $\v \lambda$.

\subsection{Properties}
We check the requirement list from Section \ref{sec:properties}:
\begin{enumerate}[(a)]
\item The Poisson reduction family is not necessarily parsimonious. The number of parameters $\dim(\theta)$ is determined by the fitting algorithm.
\item We fit the family via method of moments using a fast but non-optimal greedy algorithm.
\item We sample $\v y\sim q^\Poi_{(\mathcal S,\lambda)}$ using $\v y=\tau_{\mathcal S}(\v v)$ with $\v v\sim h_{\v \lambda}$.
\item {\color{alarmcolor} We cannot evaluate $q^\Poi_{(\mathcal S,\lambda)}(\v y)$ since it requires summation of $2^{d-\abs{\v y}}-1$ terms using an inclusion-exclusion principle which is computationally not feasible.}
\item The family $q^\Poi_{(\mathcal S,\lambda)}$ can partially reproduce the mean and certain correlation structures of the data $\m X$. {\color{alarmcolor}We cannot sample negative correlations.}
\end{enumerate}
Since the family is limited to certain patterns of non-negative correlations, we cannot use it as general-purpose family in adaptive Monte Carlo algorithms. It might be useful, however, if we know that the target distribution $\pi$ has strictly non-negative correlations.

\section{The Archimedean copula family}
\cite{genest2007primer} discuss in detail the potentials and pitfalls of applying copula theory, which is well developed for bivariate, continuous random variables, to multivariate discrete distribution. Yet, there have been earlier attempts to sample binary vectors via copulas: \citet{lee1993generating} describes how to construct an Archimedean copula, more precisely the Frank family, (see e.g. \citet[p.119]{nelsen2006introduction}), for sampling multivariate binary data.

Unfortunately, most results in copula theory do not easily extend to high dimensions. Indeed, we need to solve a non-linear equation for each component when generating a random vector from the Frank copula, and \citet{lee1993generating} acknowledges that this is only applicable for $d\leq3$. For low-dimensional problems, however, we can just enumerate the solution space $\B^d$ and draw from an alias table \citep{walker1977efficient}, which somewhat renders the Archimedean copula approach an interesting exercise, but without much practical value in Monte Carlo applications.



%


%

\onecolumn
\section{Appendix}

\begin{proof}[Proof Proposition \ref{prop:bahadur}]
\label{proof:bahadur}
Recall that $\mathcal{I}=2^D$ and $v_I(\v\gamma)=\prod_{i\in I}\lbrack(\gamma_i-m_i)/\sqrt{m_i(1-m_i)}\rbrack$ with $m_i>0$ for all $i\in D$. We define an inner product
\begin{equation*}
\textstyle
(f,g)\eqdef\evplus{q^\Prod_{\v m}}{f(\v\gamma)g(\v\gamma)}=\sum_{\v\gamma\in\B^d}f(\v\gamma)g(\v\gamma)\prod_{i\in D}m_i^{\gamma_i}(1-m_i)^{1-\gamma_i}
\end{equation*}
on the vector space of real-valued functions on $\B^d$. The set $S=\set{ v_I(\v\gamma)\mid I\in\mathcal{I}}$ is orthonormal, since
\begin{align*}
(v_I,r_J)
&=\prod_{i\in I\cap J}\mathbb E_{q^\Prod_{\v m}} \Bigg( \frac{(\gamma_i-m_i)^2}{m_i(1-m_i)} \Bigg)
 \prod_{i\in (I\cup J)\setminus(I\cap J)} \mathbb E_{q^\Prod_{\v m}} \Bigg( \frac{\gamma_i-m_i}{\sqrt{m_i(1-m_i)}} \Bigg) \\
&=
\begin{cases}
0 & \text{for } I\neq J \\
1 & \text{for } I=J,
\end{cases}
\end{align*}
There are $2^d-1$ elements in $S$ and $q^\Prod_{\v m}(\v\gamma)>0$ which implies that $S\cup\set 1$ is an orthonormal basis of the real-valued function on $\B^d$. It follows that each function $f\colon\B^d\to\R$ has exactly one representation as linear combination of functions in $S\cup\set 1$ which is $f=(f,1)+\sum_{I\in\mathcal I}v_I(f,v_I)$. Since
\begin{equation*}
\textstyle
(\pi/q^\Prod_{\v m},v_I)
=\sum_{\v\gamma\in \B^d} (\pi(\v\gamma)/q^\Prod_{\v m)(\v\gamma)}v_I(\v\gamma)q^\Prod_{\v m}(\v\gamma)
=\evplus{\pi}{v_I(\v\gamma)}=c_I,
\end{equation*}
we obtain $\pi(\v\gamma)/q^\Prod_{\v m}(\v\gamma)=\textstyle1+\sum_{I\in\mathcal{I}} v_I(\v\gamma) \ c_I$ for $f=\pi/q^\Prod_{\v m}$ which concludes the proof.
\end{proof}

%


%
\begin{proof}[Proof Proposition \ref{prop:cross moments}]
\label{proof:cross moments}
We first derive two auxiliary results to structure the proof.
\begin{proof}[Lemma 1]
For a set $I\subseteq D$ of indices it holds that
\begin{equation*}
\sum_{\v\gamma\in\B^d} \prod_{k\in I\cup\lbrace i,j\rbrace} \gamma_k\ =2^{d-\abs{I}-2+\ind_{I}(i)+\ind_{I\cup\lbrace i\rbrace}(j)}.
\end{equation*}
For an index set $M\subseteq D$, we have the sum formula $\sum_{\v\gamma\in\B^d} \prod_{k\in M} \gamma_k=2^{d-\abs{M}}$. If we have an empty set $M=\emptyset$ the sum equals $2^d$ and each time we add a new index $i\in D\setminus M$ to $M$ half of the addends vanish. The number of elements in $M=I\cup\lbrace i,j\rbrace$ is the number of elements in $I$ plus one if $i\notin I$ and again plus one if $i\neq j$ and $j\notin I$. Written using indicator function, we have $
\abs{I\cup\lbrace i,j\rbrace}
=\abs{I}+\ind_{D\setminus I}(i)+\ind_{D\setminus (I \cup\lbrace i\rbrace)}(j)
=\abs{I}+2-\ind_{I}(i)-\ind_{I\cup\lbrace i\rbrace}(j)$ which implies Lemma 1.
\end{proof}

\begin{proof}[Lemma 2]
\begin{equation*}
\sum_{i\in D} \sum_{j\in D} 2^{\ind_{I}(i)+\ind_{I\cup\lbrace i\rbrace}(j)}\ a_{i,j}=
\v 1\t \m A \v 1 + \tr{\m A} + \sum_{i\in I} \left\lbrack 2\sum_{j\in D} a_{i,j}+\sum_{j\in I\setminus\set i} a_{i,j} \right\rbrack
\end{equation*}
Straightforward calculations:
\begin{align*}
2^{\ind_{I}(i)+\ind_{I\cup\set i}(j)}
&=
(1+\ind_{I}(i))(1+\ind_{I\cup\set i}(j)) \\
&=
(1+\ind_{I}(i))(1+\ind_{I}(j)+\ind_{\set i}(j)-\ind_{I\cap\set i}(j))
\\
&=
1+\ind_{I}(i)
+\ind_{I}(j)+\ind_{I}(i)\ind_{I}(j) \\
&\qquad+\ind_{\set i}(j)+\ind_{I}(i)\ind_{\lbrace i \rbrace}(j)
-\ind_{I\cap\lbrace i\rbrace}(j)-\ind_{I}(i)\ind_{I\cap \lbrace i \rbrace}(j)
\\
&=
1+\ind_{\lbrace i\rbrace}(j)+\ind_{I}(i)+\ind_{I}(j)+\ind_{I\times I}(i,j)-\ind_{I\cap\lbrace i\rbrace}(j),
\end{align*}
where we used the identity
\begin{equation*}
\ind_{I}(i)\ind_{\lbrace i \rbrace}(j)
=\ind_{I}(i)\ind_{I}(i)\ind_{\lbrace i \rbrace}(j)
=\ind_{I}(i)\ind_{I}(j)\ind_{\lbrace i \rbrace}(j)
=\ind_{I}(i)\ind_{I\cap \lbrace i \rbrace}(j) 
\end{equation*}
in the second line. Thus, we have
\begin{align*}
&\ \sum_{i\in D} \sum_{j\in D} 2^{\ind_{I}(i)+\ind_{I\cup\lbrace i\rbrace}(j)}\ a_{i,j} \\
=&\
\sum_{i\in D} \sum_{j\in D} \left(
1+\ind_{\lbrace i\rbrace}(j)+\ind_{I}(i)+\ind_{I}(j)+\ind_{I\times I}(i,j)-\ind_{I\cap\lbrace i\rbrace}(j)
\right)\ a_{i,j} \\
=&\
 \sum_{i\in D}\sum_{j\in D} a_{i,j}
+\sum_{j\in D} a_{j,j}
+\sum_{i\in I}\sum_{j\in D} a_{i,j}
+\sum_{i\in D} \sum_{j\in I} a_{i,j}
+\sum_{i\in I}\sum_{j\in I} a_{i,j}
-\sum_{i\in I} a_{j,j} \\
=&\
\v 1\t \m A \v 1 + \tr{\m A} +
\sum_{k\in I}
\left\lbrack 2\sum_{l\in D} a_{k,l}+\sum_{l\in I} a_{k,l}-a_{k,k} \right\rbrack \\
=&\
\v 1\t \m A \v 1 + \tr{\m A} +
\sum_{k\in I}
\left\lbrack 2\sum_{l\in D} a_{k,l}+\sum_{l\in I\setminus\set k} a_{k,l} \right\rbrack
\end{align*}
The last line is the assertion of Lemma 2.
\end{proof}

\noindent
Using the two Lemmata above, we find a convenient expression for the cross-moment
\begin{align*}
m_I
&=
\sum_{\v\gamma\in\B^d} (\prod_{k\in I} \gamma_k)\ \mu(a_0+\v\gamma\t\m A \v\gamma)
&\\
&=
\mu\left\lbrack
\sum_{\v\gamma\in\B^d} a_0 + \sum_{\v\gamma\in\B^d} (\prod_{k\in I} \gamma_k) \sum_{i\in D} \sum_{j\in D} \gamma_i\gamma_j\ a_{i,j}
\right\rbrack
&\\
&=
\mu\left\lbrack
2^{d-\abs I}a_0+
\sum_{i\in D} \sum_{j\in D} a_{i,j}\ \sum_{\v\gamma\in\B^d} (\prod_{k\in I\cup\lbrace i,j\rbrace} \gamma_k)
\right\rbrack
\text{(Lemma 1)}\\
&=
\mu\left\lbrack
2^{d-\abs I}a_0+\,
\sum_{i\in D} \sum_{j\in D} 2^{d-\abs{I\cup\lbrace i,j\rbrace}}\ a_{i,j}
\right\rbrack
&\\
&=
\mu 2^{d-\abs{I}-2}
\left\lbrack 4a_0+
\sum_{i\in D} \sum_{j\in D} 2^{\ind_{I}(i)+\ind_{I\cup\lbrace i\rbrace}(j)}\ a_{i,j}
\right\rbrack
\text{(Lemma 2)}\\
&=
\mu 2^{d-\abs{I}-2}\left\lbrack 4a_0+
\v 1\t \m A \v 1 + \tr{\m A} +
\sum_{i\in I}
\left\lbrack 2\sum_{j\in D} a_{i,j}+\sum_{j\in I\setminus\set i} a_{i,j} \right\rbrack
\right\rbrack
\end{align*}
Since $m_\emptyset=1$ by definition, we the normalizing constant is
\begin{equation*}
\mu=2^{-d+2}\left(4a_0+\v 1\t \m A \v 1 + \tr{\m A}\right)^{-1},
\end{equation*}
which allows us to write down the normalized cross-moments
\begin{align*}
m_I=\frac{1}{2^{\abs I}}+
\frac{
\sum_{i\in I} \left\lbrack 2\sum_{j\in D} a_{i,j}+\sum_{j\in I\setminus\set i} a_{i,j} \right\rbrack
}{
2^{\abs I}(4a_0+\v 1\t \m A \v 1 + \tr{\m A})
}.
\end{align*}
The proof is complete.
\end{proof}

%


%
\begin{proof}[Proof Proposition \ref{prop:marginals}]
\label{proof:marginals}
We margin out the last component $d$. Let $I=\set{1,\dots,d-t}$,
\begin{align*}
q_{\m A, a_0}^{(d-1)}(\v\gamma_I)\mu^{-1}
&=
\left( q_{\m A, a_0}^{(d)}(\v\gamma_I,1)+q_{\m A, a_0}^{(d)}(\v\gamma_I,0) \right)\mu^{-1} \\
&=
2a_0+(\v\gamma_I,1)\t \m A (\v\gamma_I,1) + (\v\gamma_I,0)\t \m A (\v\gamma_I,0) \\
&=
2a_0+ \tr{\m A \left\lbrack (\v\gamma_I,1)(\v\gamma_I,1)\t + (\v\gamma_I,0)(\v\gamma_I,0)\t \right\rbrack} \\
&=
2a_0 + \tr{\m A
\begin{bmatrix}
2 \v\gamma_I\v\gamma_{I}\t & \v\gamma_{I} \\
  \v\gamma_I\t             & 1
\end{bmatrix}
}
\end{align*}
Iterating the argument, we obtain for $I=\set{1,\dots,d-t}$ and $I^c=D\setminus I$
\begin{align*}
q_{\m A, a_0}^{(d-t)}(\v\gamma_I)\mu^{-1}
&=2^ta_0+2^{t-2}\,
\tr{\m A
\begin{bmatrix}
 4\, \v\gamma_I\v\gamma_I\t & 2\, \v\gamma_I\v 1_{t}\t \\
 2\, \v 1_{t}\v\gamma_I\t   & \v1_{t}\v1_{t}\t+\m I_{t}
\end{bmatrix}
}
\end{align*}
Straightforward calculations:
\begin{align*}
&\ \tr{\m A
\begin{bmatrix}
 4\, \v\gamma_I\v\gamma_I\t & 2\, \v\gamma_I\v 1_{t}\t \\
 2\, \v 1_{t}\v\gamma_I\t   & \v1_{t}\v1_{t}\t+\m I_{t}
\end{bmatrix}} \\
=&\ \tr{
\m A
\left\lbrack
(2\,\v\gamma_I,\v 1_{t})(2\,\v\gamma_I,\v 1_{t})\t +
\diag{\v 0_I,\v 1_{t}}
\right\rbrack}
\\
=&\
\left\lbrack
(2\,\v\gamma_I,\v 1_{t})\t\m A(2\,\v\gamma_I,\v 1_{t}) + \tr{\m A \diag{\v 0_I,\v 1_{t})}}
\right\rbrack
\\
=&\
\left\lbrack
4\sum_{i\in I}\sum_{j\in I} \gamma_i\gamma_j a_{i,j}
+ 4 \sum_{i\in I}\sum_{j\in I^c} \gamma_i a_{i,j}
+ \sum_{i\in I^c} \sum_{j\in I^c} a_{i,j}
+ \sum_{i\in I^c} a_{i,i}
\right\rbrack
\\[.5em]
=&\
\left\lbrack
4\sum_{i\in I}\gamma_i(\sum_{j\in I} \gamma_j a_{i,j}
+ \sum_{j\in I^c} a_{i,j})
+ \sum_{i\in I^c} \sum_{j\in I^c} a_{i,j}
+ \sum_{i\in I^c} a_{i,i}
\right\rbrack
\end{align*}
The proof is complete.
\end{proof}

%


%
\begin{proof}[Proof Proposition \ref{prop:approx marginals}]
\label{proof:approx marginals}
For convenience of notation, let $\v\gamma_{-}=(\gamma_1,\dots,\gamma_{d-1})$. Note that $q_{\m A}(\v\gamma)=\mu\exp(\v\gamma_{-}\t\m A' \v\gamma_{-}+\gamma_d(2\v b\t\v\gamma_{-}+c))$. The marginal distribution is therefore
\begin{align*}
\pi(\v\gamma_{-})
&=\mu\,\exp(\v\gamma_{-}\t\m A'\v\gamma_{-})\left(1+\exp(2\v\gamma_{-}\t\v b+c)\right)\\
&=\mu\,\exp\left(\v\gamma_{-}\t\m A'\v\gamma_{-}+\v\gamma_{-}\t\v b+\frac{c}{2}\right)
\left(\exp(-\v\gamma_{-}\t\v b-\frac{c}{2})+\exp(\v\gamma_{-}\t\v b+\frac{c}{2})\right) \\
&=\mu\,\exp\left(\v\gamma_{-}\t\m A'\v\gamma_{-}+\v\gamma_{-}\t\v b+\frac{c}{2}\right)\,2\cosh\left(\v\gamma_{-}\t\v b+\frac{c}{2}\right).
\end{align*}
The marginal log mass function is thus
\begin{equation*}
\log \pi(\v\gamma_{-})=
\log(2\mu)+\frac{c}{2}+\v\gamma_{-}\t\m A'\v\gamma_{-}+\v\gamma_{-}\t\v b+\log\cosh\left(\v\gamma_{-}\t\v b+\frac{c}{2}\right).
\end{equation*} 
For $\log\cosh$ we can use a Taylor approximation
\begin{align*}
\log\cosh(\v\gamma_{-}\t\v b+\frac{c}{2})\approx
\log\cosh(\frac{c}{2})+
\v\gamma_{-}\t\v b\,\tanh(\frac{c}{2})+
\frac{1}{2}\,(\v\gamma_{-}\t\v b)^2\,\mathrm{sech}^2(\frac{c}{2})
\end{align*} 
to obtain
\begin{align*}
\log \pi(\v\gamma_{-})
&\approx
\log(2\mu\cosh(\frac{c}{2}))+\frac{c}{2}+\v\gamma_{-}\t\m A'\v\gamma_{-} \\
&\qquad +\big(1+\tanh(\frac{c}{2})\big)\v\gamma_{-}\t\v b+ \frac{1}{2}\,\mathrm{sech}^2(\frac{c}{2})(\v\gamma_{-}\t\v b)^2
\end{align*}
Since $\v\gamma_{-}$ is a binary vector, we have $\v\gamma_{-}\t\v b=\v\gamma_{-}\t\diag{\v b}\v\gamma_{-}$ and can thus rewrite the inner products as
\begin{align*}
\v\gamma_{-}\t\m A'\v\gamma_{-}+\gamma_{-}\t\v b+(\v\gamma_{-}\t\v b)^2
&=\tr{\m A'\v\gamma_{-}\v\gamma_{-}\t+\diag{\v b}\v\gamma_{-}\v\gamma_{-}\t+\v b \v b\t \v\gamma_{-}\v\gamma_{-}\t} \\
&=\v\gamma_{-}\t(\m A'+\diag{\v b}+\v b \v b\t) \v\gamma_{-}.
\end{align*} 
We let denote
\begin{equation*}
\mu^*=2\mu\cosh(\frac{c}{2})\exp(\frac{c}{2})=\mu (\exp(-\frac{c}{2})+\exp(\frac{c}{2}))\exp(\frac{c}{2})=\mu\,(1+\exp(c))
\end{equation*}
and
\begin{equation*}
\m A^*=
\m A'
+\left(1+\tanh(\frac{c}{2})\right)\diag{\v b}
+\frac{1}{2}\,\mathrm{sech}^2(\frac{c}{2})\v b \v b\t \\[.5em]
\end{equation*}
to form the approximation $\pi(\v\gamma_{-})\approx\mu^*\,\exp(\v\gamma_{-}\t\m A^*\v\gamma_{-})$ which completes the proof.\\
\end{proof}

%


%
\begin{proof}[Proof Proposition \ref{prop:poisson family}]
\label{proof:poisson family}
Straightforward calculations using an inclusion-exclusion argument for the union of events:
\begin{align*}
q_{(\mathcal S,\lambda)}(\v\gamma)
&=\sum_{\v v \in \tau^{-1}(\v\gamma)}h_{\lambda}(\v v)
\textstyle =\probplus{h_{\lambda}}{\cap_{i\in D}\set{\ind_{\set 0} \sum_{k\in S_i} v_k=\gamma_i}} \\
&\textstyle =\probplus{h_{\lambda}}{
\cap_{i\in D_1} \cap_{k\in S_i} \set{v_k=0},\ 
\cap_{i\in D_0} \cup_{k\in S_i} \set{v_k>0}} \\ 
&=
\probplus{h_{\lambda}}{\cap_{i\in D_1} \cap_{k\in S_i} \set{v_k=0}}
\probplus{h_{\lambda}}{\cap_{i\in D_0} \cup_{k\in S_i\setminus \cup_{j\in D_1}S_j} \set{v_k>0}} \\ 
&=
\probplus{q_{(\mathcal S,\lambda)}}{\v\gamma_{D_1}=\v 1}
\left(
1-\probplus{h_{\lambda}}{\cup_{i\in D_0} \cap_{k\in S_i\setminus \cup_{j\in D_1}S_j} \set{v_k=0}}
\right) \\ 
&= m_{D_0}\left\lbrack1- \sum_{t=1}^{\abs{D_0}}(-1)^{t-1} \sum_{I\subseteq \mathcal I_t}
\probplus{h_{\lambda}}{\cap_{i\in I} \cap_{k\in S_i\setminus \cup_{j\in D_1}S_j} \set{v_k=0}}\right\rbrack \\
&= m_{D_0}\left\lbrack1- \sum_{t=1}^{\abs{D_0}}(-1)^{t-1} \sum_{I\subseteq \mathcal I_t}
\exp\left(\sum_{k\in \cap_{i\in I} S_i\setminus \cup_{j\in D_1}S_j} -\lambda_k\right) \right\rbrack.
\end{align*}
The proof is complete.
\end{proof}

\end{document}